%% file: main.tex
\documentclass[11pt,a4paper,leqno]{amsart}

\usepackage[utf8]{inputenc}
\usepackage[T1]{fontenc}
\usepackage{amsmath,amssymb,amsthm}
\usepackage[margin=2.35cm]{geometry}
\usepackage[expansion=false]{microtype}
\usepackage{enumitem}
\usepackage{mathtools}
\usepackage[dvipsnames]{xcolor}
\usepackage{tikz}
\usepackage{crossreftools}
\usepackage{caption}
\usepackage{hyperref}
\definecolor{darkspringgreen}{rgb}{0.09,0.6,0.1}

\newcommand{\gw}[1]{\left\langle #1 \right\rangle}
\providecommand{\h}{h}
\renewcommand{\AA}{\mathbb{A}}
\input{figures/drawing_tools}

\theoremstyle{plain}
\newtheorem{theorem}{Theorem}[section]
\newtheorem{proposition}[theorem]{Proposition}
\newtheorem{lemma}[theorem]{Lemma}
\newtheorem{corollary}[theorem]{Corollary}
\theoremstyle{definition}
\newtheorem{definition}[theorem]{Definition}
\newtheorem{notation}[theorem]{Notation}
\newtheorem{setting}[theorem]{Setting}
\theoremstyle{remark}
\newtheorem{remark}[theorem]{Remark}

\newcommand{\GW}{\mathrm{GW}}
\newcommand{\ang}[1]{\langle #1 \rangle}
\newcommand{\bb}{\mathbb}
\newcommand{\cD}{\mathcal{D}}
\newcommand{\cF}{\mathcal{F}}
\newcommand{\GWuniv}{\mathrm{GW}^{\mathrm{univ}}}
\newcommand{\cS}{\mathcal{S}}
\newcommand{\cT}{\mathcal{T}}
\newcommand{\mult}{\mathrm{mult}}
\newcommand{\sgn}{\mathrm{sgn}}
\newcommand{\Atrop}{\bb{A}^1}
\newcommand{\Ntrop}{N^{\Atrop,\mathrm{trop}}}
\newcommand{\Nfloor}{N^{\Atrop,\mathrm{floor}}}

\title[Merge-position invariance in enriched floor diagrams]
      {Merge-position invariance in quadratically enriched
       tropical floor diagrams}
\author{Yanis Hedjem}
\address{D\'epartement de math\'ematiques et applications,
         \'Ecole normale sup\'erieure -- PSL, Paris, France}
\email{yhedjem@clipper.ens.psl.eu}
\date{First arXiv submission: May 9, 2026; revised September 28, 2026}


\hypersetup{
  hidelinks,
  pdftitle={Merge-position invariance in quadratically enriched tropical floor diagrams},
  pdfauthor={Yanis Hedjem},
  pdfsubject={Tropical enumerative geometry},
  pdfkeywords={tropical enumerative geometry, A1-enumerative geometry, Grothendieck-Witt ring, floor diagrams, Welschinger invariants}
}

\begin{document}

\begin{abstract}
Jaramillo Puentes et al. define a Grothendieck--Witt-valued count
by floor diagrams for rational curves in smooth toric del Pezzo
surfaces with simple point conditions and quadratic double point
conditions. We prove that, with the quadratic parameters kept in a
fixed order, the count does not depend on which adjacent pairs form
the double conditions.

Given two merge configurations, we factor the difference between
their floor counts as
\[
  \Delta N=C\prod_{j=1}^s(\ang{d_j}-\ang{1}).
\]
The coefficient~$C$ can be lifted to a free rank-three ring,
independently of the field. Writing this lift as
$n_1\ang{1}+n_2\ang{2}+\mu h$, invariance of the real broccoli
count yields $n_2=-n_1$. The twin-tree parity condition enters
through a residual calculation, which shows that $n_1$ is even. The relation
$2(\ang{1}-\ang{2})=0$ now makes the wall-crossing vanish. The
floor count is therefore independent of merge positions over each
field to which the correspondence theorem applies.
\end{abstract}

\maketitle

\section{Introduction}\label{sec:intro}

\subsection{Quadratically enriched tropical counts}
\label{sub:qe}

Over the complex numbers, one counts rational curves of a prescribed
toric degree through the expected number of point conditions.
For a smooth toric del Pezzo polygon~$\Delta$, we denote this
genus~$0$ invariant by~$N_\Delta$.
Over~$\bb{R}$ the unsigned count need not be invariant, whereas
Welschinger~\cite{Wel03, Wel05} proved invariance of the signed
count $W_{\Delta,s}$ for real point conditions together with~$s$
pairs of complex conjugate conditions.

In the $\bb{A}^1$-enumerative refinements due to
Levine~\cite{Lev18}, Pauli--Wickelgren~\cite{PW21}, and
Kass--Levine--Solomon--Wickelgren~\cite{KLSW23}, these integers
are replaced by classes in the Grothendieck--Witt ring~$\GW(k)$,
where~$k$ is the ground field.
Jaramillo~Puentes, Markwig, Pauli, and R\"ohrle
\cite{JMPR25} construct a tropical model of these counts for
simple point conditions and double conditions over the quadratic
field extensions $k(\sqrt{d_1}),\ldots,k(\sqrt{d_s})$, as well as split
specialisations. Under the rank map, their tropical count recovers
Mikhalkin's complex count~\cite{Mik05}; over~$\bb{R}$, its
signature recovers Shustin's tropical Welschinger
count~\cite{Shu06}. These comparisons are proved in
\cite[Corollary~5.7 and Proposition~5.13]{JMPR25}.
Recent computations and structural results for quadratically
enriched curve counts include, for example,
Brugall\'e--Wickelgren~\cite{BW25} and
Brugall\'e--Rau--Wickelgren~\cite{BRW25}.

\begin{theorem}[Jaramillo Puentes--Markwig--Pauli--R\"ohrle,
  tropical correspondence]
\label{thm:JMPR}
Let~$\Delta$ be a smooth toric del Pezzo polygon and assume
$r+2s=n(\Delta)=|\partial\Delta\cap\bb{Z}^2|-1$.  Over a
perfect field satisfying the characteristic assumptions of
\cite{JMPR25}, and for quadratic field extensions as in
\cite[Setting~1.1]{JMPR25}, the algebraic quadratically
enriched count satisfies
\[
  N_\Delta^{\Atrop}(r,(d_1,\dots,d_s))
  =\Ntrop_\Delta(r,(d_1,\dots,d_s))
  \coloneqq\sum_\Gamma\mult^{\Atrop}(\Gamma),
\]
where the sum is over rational tropical stable maps of
degree~$\Delta$ through vertically stretched conditions.
Moreover,
$\Ntrop_\Delta(r,(d_1,\dots,d_s))
=\Nfloor_\Delta(r,(d_1,\dots,d_s))$
by the floor-diagram formula of~\cite[Theorem~10.13]{JMPR25}.
\end{theorem}

The tropical curves are parametrised stable maps: the abstract
graph is part of the data, including the twin-tree structures
recalled in Subsection~\ref{subsub:twin-trees}.
When a parameter is a square, the split case is recovered by
\cite[Proposition~5.4]{JMPR25}.

\subsection{The merge-position problem}
\label{sub:merge-intro}

For a lattice polygon~$\Delta$, set
\[
  n(\Delta)\coloneqq |\partial\Delta\cap\bb{Z}^2|-1.
\]
For a smooth toric del Pezzo polygon this is the expected
number of point conditions for rational curves of degree
~$\Delta$, and it is the notation used in~\cite{JMPR25}.

The floor-diagram formula starts from a vertically ordered
configuration of $r+2s$ simple point conditions and then
merges~$s$ adjacent pairs into double point conditions.
The data of which adjacent pairs are merged are the
\emph{merge positions}.  The algebraic count on the left of
Theorem~\ref{thm:JMPR} is independent of such auxiliary
choices.  In the floor-diagram product formula this invariance
is hidden: changing a merge position changes the allowed floor
diagrams and the local multiplicity factors in a way that is
not easy to control directly.

For instance, if $s=1$, a configuration of
$r+2$ ordered simple positions has possible merge positions
$(1,2),(2,3),\ldots,(r+1,r+2)$.  The theorem asserts that the
sum of the enriched floor multiplicities is the same for all
these choices.  A unit shift replaces, say, the pair $(i,i+1)$
by $(i+1,i+2)$; the proof reduces invariance under arbitrary
changes of merge positions to these elementary shifts.

The motivation for proving such an invariance tropically is the
same as in earlier developments of tropical enumerative
geometry.  Mikhalkin's correspondence theorem made complex
enumerative invariance visible through tropical counts
\cite{Mik05}; Gathmann--Markwig then gave a direct proof,
entirely within tropical geometry, that the relevant tropical
curve counts are independent of the point positions~\cite{GM07}.
For real counts, Gathmann--Markwig--Schroeter introduced
broccoli curves and proved, in the toric del Pezzo case, a
tropical invariance statement for Welschinger counts with real
and conjugate point conditions~\cite{GMS13}. Tropical methods also
yield computations and recurrences, including the Caporaso--Harris
type formula for Welschinger invariants of real toric del Pezzo
surfaces with real point conditions~\cite{IKS09}. Tropical
wall-crossing arguments establish the invariance of the
Block--G\"ottsche refined tropical curve counts
\cite{IM13,BG16}.  Thus making a known algebraic invariance
visible inside a tropical formula is useful even when an
abstract algebraic reason for the invariance is already
available.

\subsection{Main theorem}
\label{sub:main}

\begin{theorem}[Main theorem]
\label{thm:main-intro}
For every Newton polygon~$\Delta$ of a smooth toric del Pezzo
surface, every perfect field~$k$ with $\mathrm{char}(k)\neq2$
(and, in positive characteristic, greater than both~$3$ and
the diameter of~$\Delta$), every
$d_1,\dots,d_s\in k^\times$, and every
$r\geq0$ with $r+2s=n(\Delta)$,
the floor-diagram count
\[
  \Nfloor_\Delta(r,(d_1,\dots,d_s))\in\GW(k)
\]
is independent of the merge positions, with the ordered tuple
$(d_1,\dots,d_s)$ fixed.
If all~$d_j$ are non-squares, the double conditions are those
over the quadratic field extensions $k(\sqrt{d_j})$ of
\cite[Setting~1.1]{JMPR25}.  If some~$d_j$ is a square, the
corresponding condition is interpreted through the split
specialisation of \cite[Prop.~5.4]{JMPR25}, equivalently as the
dissolution of that double condition into two adjacent simple
conditions.
\end{theorem}

\subsection{Strategy of the proof}
\label{sub:strategy}

Fix two merge configurations and let~$\Delta N$ be the
difference of the corresponding floor-diagram sums.  The
tropical part of the proof is a rigidity argument in~$\GW(k)$,
performed directly on the floor-diagram product formula.

With the fixed normal form for the local factors used in
Sections~\ref{sec:floor}--\ref{sec:typeA}, the obstruction has
the form
\[
  \Delta N
  =C\prod_{j=1}^s(\ang{d_j}-\ang{1}),
\]
where the coefficient~$C$ is the image of a universal element
\[
  \widetilde C=n_1\ang{1}+n_2\ang{2}+\mu h
  \in \GWuniv
  =\bb{Z}[V_4]/(\ang{2}+\ang{-2}-h).
\]
The integers $n_1,n_2,\mu$ depend only on
$\Delta,r,s$ and the choice of the two merge configurations,
not on the field~$k$ or the parameters~$d_j$.
The broccoli-curve proof of tropical invariance for real
Welschinger counts~\cite{GMS13} gives $n_2=-n_1$.
After reduction modulo~$2$ and~$h$, we use the parity condition
on twin trees to prove that $n_1$ in the universal lift is even.
Using the elementary relation $2(\ang{1}-\ang{2})=0$, we conclude
that $\Delta N=0$. For the real comparison, see
Section~\ref{sec:identities}; the calculation after reduction
is in Section~\ref{sec:residual}.

Section~\ref{sec:gw} contains the Grothendieck--Witt identities
and the definition of~$\GWuniv$.
The local factors of the floor-diagram multiplicity are written
out in Section~\ref{sec:floor}. In
Sections~\ref{sec:typeA}--\ref{sec:linear}, we establish the
type-A identities, dissolution, real specialisation,
localisation, and multi-affinity used in the proof.
Section~\ref{sec:residual} proves rigidity after reduction;
Section~\ref{sec:main} finishes the argument.

\subsection{Conventions and terminology}
\label{sub:conventions}

Throughout, $\mathrm{char}(k)\neq 2$.
We call a field \emph{admissible} when it meets the
hypotheses of Theorem~\ref{thm:main-intro} and
Setting~\ref{set:JMPR}.
An edge weight~$m$ is always a positive integer;
we write $c=\frac{m-1}{2}$ for odd~$m$. In universal
expressions, the integer coefficient of~$h$ is denoted by~$\mu$.
The use of a rank-one symbol involving~$m$ presupposes a
non-zero argument in~$k$. In particular, identities with
$\ang{m}$ are understood to require $m\cdot1_k\neq0$.
Geometric edge weights satisfy this requirement by the
characteristic bound in Setting~\ref{set:JMPR}. Indeed, the
weight is at most the lattice length of the edge dual to its
image in the Newton subdivision of~$\Delta$, and that length
is at most the diameter of~$\Delta$.

By a ``simple point'' we mean a single point condition, and by a
``double point'' a merged adjacent pair with quadratic
parameter~$d_j$.
In a \emph{unit shift}, one pair moves by one position in the
vertical order while all other merged pairs stay fixed.
We say that the conditions are \emph{vertically stretched} if
the vertical gaps are large enough to allow the tropical curves
to decompose into floors. For the exact genericity and stretching
hypotheses we follow \cite[Definitions~3.6 and~3.15]{JMPR25};
the floor decomposition and the bounds on weights are given in
\cite[Lemma~3.17 and Proposition~3.18]{JMPR25}.

When comparing two merge configurations, we call the difference
of their floor-diagram sums the \emph{wall-crossing difference}.
Our reference for~\cite{JMPR25} is arXiv:2502.02569v4
(June 2026). The theorem numbers and local multiplicity formulas
used here are those of this version.

\section{The Grothendieck--Witt ring and a universal model}
\label{sec:gw}

In the Grothendieck--Witt ring, non-degenerate symmetric bilinear
forms are considered up to the usual stable isometry relations.
Quadratically enriched counts lie in this ring; their rank is
the ordinary complex count, while over~$\bb{R}$ their signature
is the signed real count.
We recall here just the elementary identities used in the proof;
for quadratic forms and the Witt ring we refer to
\cite[Chaps.~I--II]{Lam05}, and for their role in
$\bb{A}^1$-enumerative geometry to
\cite{Lev18,PW21,KLSW23}.

\begin{notation}\label{not:gw}
For $a\in k^\times$, write $\ang{a}$ for the class in
$\GW(k)$ of the rank-$1$ symmetric bilinear form
$(x,y)\mapsto axy$.  Every class in $\GW(k)$ is a
$\bb{Z}$-linear combination of such one-dimensional forms, and
the symbols satisfy the multiplicative relations
\begin{equation}\label{eq:gw-rels}
  \ang{a}\ang{b}=\ang{ab},\qquad
  \ang{a}=\ang{ab^2}\qquad(a,b\in k^\times).
\end{equation}
The \emph{hyperbolic form} is
$h\coloneqq\ang{1}+\ang{-1}$, of rank~$2$.
\end{notation}

\begin{lemma}[Arithmetic of $\GW(k)$]\label{lem:gw-arith}
The following hold in $\GW(k)$.
\begin{enumerate}[label=\textup{(\roman*)},nosep]
\item\label{it:hyp}
  $\ang{a}+\ang{-a}=h$ for all $a\in k^\times$.
\item\label{it:absorb}
  $h\cdot\ang{a}=h$ for all $a\in k^\times$.
\item\label{it:hsq}
  $h^2=2h$.
\item\label{it:annihilate}
  $h\cdot(\ang{d}-\ang{1})=0$ for all $d\in k^\times$.
\end{enumerate}
\end{lemma}

\begin{proof}
The binary form $\ang{a}+\ang{-a}$ is hyperbolic: if
$e,f$ are the two diagonal basis vectors, then
$u=e+f$ and $v=(e-f)/(2a)$ are isotropic and satisfy
$B(u,v)=1$.  Thus they form a hyperbolic basis, proving
\ref{it:hyp}.
Multiplying the identity $h=\ang{1}+\ang{-1}$ by~$\ang{a}$
gives $h\ang{a}=\ang{a}+\ang{-a}=h$, proving~\ref{it:absorb}.
Then $h^2=h(\ang{1}+\ang{-1})=h+h=2h$, and
$h(\ang{d}-\ang{1})=h-h=0$.
\end{proof}

\begin{remark}[The ring $\GW(\bb{R})$]\label{rem:gw-R}
Over~$\bb{R}$, $\ang{a}=\ang{1}$ if $a>0$ and
$\ang{a}=\ang{-1}$ if $a<0$; the \emph{signature}
$\sgn\colon\GW(\bb{R})\to\bb{Z}$ defined by
$\sgn(\ang{a})=\mathrm{sign}(a)$ is a ring homomorphism,
with $\sgn(h)=0$.  In particular $\ang{2}=\ang{1}$ over
$\bb{R}$.
\end{remark}

\begin{lemma}[$2$-torsion identity]\label{lem:two-torsion}
For every field~$k$ with $\mathrm{char}(k)\neq2$ and every
$a\in k^\times$, one has $2(\ang{a}-\ang{2a})=0$ in~$\GW(k)$.
In particular, $2(\ang{1}-\ang{2})=0$.
\end{lemma}

\begin{proof}
Witt's relation
$\ang{x}+\ang{y}=\ang{x+y}+\ang{xy(x+y)}$, valid for
$x,y,x+y\in k^\times$
(\cite[Chap.~II, \S4, relation~(R03)]{Lam05}), with $x=y=a$ gives
$2\ang{a}=\ang{2a}+\ang{2a^3}=2\ang{2a}$, so
$2(\ang{a}-\ang{2a})=0$ for every $a\in k^\times$.
Taking $a=1$ gives $2(\ang{1}-\ang{2})=0$.
\end{proof}

\begin{definition}[Universal Grothendieck--Witt coefficient
  ring]\label{def:Guniv}
Let $V_4\coloneqq\bb{Z}/2\oplus\bb{Z}/2$ be the Klein
four-group with formal generators denoted by $\ang{-1}$ and
$\ang{2}$. These symbols need not have independent square-class
images in a particular field. The integral group ring
$\bb{Z}[V_4]=\bb{Z}[\ang{-1},\ang{2}]/(\ang{-1}^2-1,\ang{2}^2-1)$
has $\bb{Z}$-basis $\{\ang{1},\ang{-1},\ang{2},\ang{-2}\}$
with $\ang{-2}\coloneqq\ang{-1}\ang{2}$.
Set $h\coloneqq\ang{1}+\ang{-1}$.  Define
\[
  \GWuniv\coloneqq
  \bb{Z}[V_4]\Big/
  \bigl(\ang{2}+\ang{-2}-h\bigr),
\]
the quotient imposing the \emph{hyperbolic relation}
$\ang{2}+\ang{-2}=h$.
For each field~$k$ with $\mathrm{char}(k)\neq2$, the
identity $\ang{2}+\ang{-2}=h$ holds in~$\GW(k)$
(Lemma~\ref{lem:gw-arith}\ref{it:hyp}), so the assignment
$\ang{a}\mapsto\ang{a}_k$ for
$a\in\{\pm1,\pm2\}$ defines a unique ring homomorphism
\[
  \varphi_k\colon\GWuniv\longrightarrow\GW(k),
  \qquad\ang{a}\longmapsto\ang{a}_k,\quad
  h\longmapsto h_k.
\]
\end{definition}

\begin{lemma}[Basic properties of $\GWuniv$]
\label{lem:GWuniv-props}
\mbox{}
\begin{enumerate}[label=\textup{(\arabic*)},nosep]
\item\label{it:GWuniv-basis}
  $\GWuniv$ is a free $\bb{Z}$-module of rank~$3$ with basis
  $\{\ang{1},\ang{-1},\ang{2}\}$; equivalently, with basis
  $\{\ang{1},h,\ang{2}\}$ one has $\ang{-1}=h-\ang{1}$ and
  $\ang{-2}=h-\ang{2}$.
\item\label{it:GWuniv-map}
  For every field~$k$ with $\mathrm{char}(k)\neq2$ the
  ring homomorphism $\varphi_k$ of
  Definition~\ref{def:Guniv} is well-defined and satisfies
  $\varphi_k(h)=h_k$.
\end{enumerate}
\end{lemma}

\begin{proof}
For~\ref{it:GWuniv-basis}, put $a=\ang{-1}$, $b=\ang{2}$ and
$\rho=b+ab-1-a$ in $\bb{Z}[V_4]$. Since
$a\rho=\rho$, $b\rho=-\rho$ and $ab\rho=-\rho$,
the ideal $(\rho)$ is exactly the additive subgroup $\bb{Z}\rho$.
Its generator has primitive coordinates $(-1,-1,1,1)$ in the
basis $(1,a,b,ab)$. Thus the quotient is free of rank~$3$,
with $ab=1+a-b$, and the claimed bases follow.

Assertion~\ref{it:GWuniv-map} follows from the definition, since
the relations
$\ang{-1}_k^2=\ang{2}_k^2=\ang{1}_k$ and
$\ang{2}_k+\ang{-2}_k=h_k$ all hold in $\GW(k)$
(Lemma~\ref{lem:gw-arith}\ref{it:hyp}).
\end{proof}

\begin{remark}\label{rem:GWuniv-nokill}
None of the homomorphisms $\varphi_k$ is injective.
Indeed, $2(\ang{2}-\ang{1})$ is zero in every~$\GW(k)$ by
Lemma~\ref{lem:two-torsion}, but is non-zero in~$\GWuniv$:
its coordinates are $(-2,0,2)$ in the basis
$\{\ang{1},\ang{-1},\ang{2}\}$.
Thus the integer coordinates of a chosen universal lift are
well-defined, but cannot in general be recovered from its
image in~$\GW(k)$.
\end{remark}

\begin{lemma}[Multi-affine cascade]\label{lem:multiaffine-cascade}
Let $R$ be a commutative ring and set
$M_s=R[x_1,\dots,x_s]/(x_1^2-1,\dots,x_s^2-1)$.
Every $F\in M_s$ has a unique expansion
\[
  F=\sum_{I\subseteq\{1,\dots,s\}} a_I x_I,
  \qquad x_I=\prod_{i\in I}x_i,\quad a_I\in R.
\]
Fix an order $j_1,\dots,j_s$ of the variables and set
$F_0=F$.  Recursively, for $q=1,\dots,s$, write the current
multi-affine expression uniquely as
\[
  F_{q-1}=U_q+x_{j_q}V_q,\qquad
  S_q=U_q+V_q,\qquad F_q=V_q,
\]
inside the remaining multi-affine algebra.  Then
\[
  F=\sum_{q=1}^{s} S_q\prod_{p<q}(x_{j_p}-1)
    +a_{\{1,\dots,s\}}\prod_{p=1}^{s}(x_{j_p}-1).
\]
Consequently, any ring homomorphism $\psi\colon M_s\to B$
satisfying $\psi(S_q\prod_{p<q}(x_{j_p}-1))=0$ for all~$q$ gives
\[
  \psi(F)=\psi(a_{\{1,\dots,s\}})
  \prod_{p=1}^{s}\bigl(\psi(x_{j_p})-1\bigr).
\]
\end{lemma}

\begin{proof}
Successive division by the monic polynomials $x_i^2-1$ gives
the stated free $R$-basis, with no assumption that $R$ is a domain.
The identity
$U_q+x_{j_q}V_q=(U_q+V_q)+V_q(x_{j_q}-1)$ gives the displayed
decomposition by iteration.  At each step $V_q$ is the
coefficient of $x_{j_q}$ in the current multi-affine
expression; after all variables have been treated, the
remaining coefficient is exactly $a_{\{1,\dots,s\}}$.
Applying~$\psi$ gives the final assertion.
\end{proof}

\section{Floor diagrams and the vertex-product multiplicity}
\label{sec:floor}

We review the constructions of~\cite[\S\S4 and~10]{JMPR25}
that we need.

\begin{setting}[{\cite[Setting~1.1]{JMPR25}}]
\label{set:JMPR}
Fix a smooth toric del Pezzo polygon~$\Delta$, an admissible perfect
field~$k$, and $r,s\geq0$ with $r+2s=n(\Delta)$. For
$j=1,\ldots,s$, let $d_j\in k^\times$ and let $Q_j$ be the associated
quadratic double condition; let $P_1,\ldots,P_r$ be simple conditions.
On the tropical side, choose a generic vertically stretched
configuration of these simple and double conditions in the sense of
\cite[Definitions~3.6 and~3.15]{JMPR25}.
\end{setting}

If $d_j$ is a square, its multiplicity is specialised to
$\ang{d_j}=\ang{1}$ and the double condition is dissolved into two
adjacent simple conditions as in~\cite[Proposition~5.4]{JMPR25}; the
associated quadratic étale algebra is $k[x]/(x^2-d_j)\simeq k\times k$.

\begin{definition}[Floor diagram,
  {\cite[Definition~10.1]{JMPR25}}]\label{def:floor}
We use precisely the rational floor diagrams of
\cite[Definition~10.1]{JMPR25}, including their bipartite-tree,
ordering, weight, divergence, leak and end conditions. Merged pairs
and twin trees are as in \cite[Lemma~10.7 and Definition~10.8]{JMPR25}.
These are the floor-shrunk types of vertically stretched
parametrised rational tropical stable maps. A merge configuration specifies the adjacent pairs
of ordered point positions treated as double conditions. We
write~$\cF_\Delta$ for the finite set of underlying floor diagrams.
\end{definition}

\begin{definition}[$\bb{A}^1$-weight of an edge,
  {\cite[Notation~4.1]{JMPR25}}]\label{def:mA1}
For a positive integer~$m$ with $m\cdot1_k\neq0$, define
\begin{equation}\label{eq:mA1}
  m^{\Atrop}
  =\begin{cases}
    \ang{m}+\dfrac{m-1}{2}\,h & m\text{ odd},\\[6pt]
    \dfrac{m}{2}\,h           & m\text{ even}.
  \end{cases}
\end{equation}
\end{definition}

\begin{definition}[Grouped type-A factor]\label{def:gamma}
For $m\geq1$ with $m\cdot1_k\neq0$ and $d\in k^\times$, we use
the type-A vertex factor of
\cite[Notation~4.1 and Definition~4.2]{JMPR25}:
\begin{equation}\label{eq:gamma}
  \gamma(m,d)
  =\begin{cases}
    \ang{m}+\dfrac{m-1}{2}\bigl(\ang{2m}+\ang{-2dm}\bigr)
      & m\text{ odd},\\[8pt]
    \dfrac{m}{2}h & m\text{ even}.
  \end{cases}
\end{equation}
The elevator-grouped formula of \cite[Remark~4.10]{JMPR25}
combines this vertex factor with the contribution $m^{\Atrop}$
at the other end of its elevator; we write
$G_A(m,d):=\gamma(m,d)m^{\Atrop}$ for the resulting contribution.
Its explicit form is computed in Proposition~\ref{prop:typeA}.
\end{definition}

\subsection{Local building blocks}
\label{sub:blocks}

We recall from \cite[\S3]{JMPR25} the \emph{vertex types} for
double point conditions and the associated local factors needed
to compute the multiplicity~\eqref{eq:mult}.  The tropical
curves are parametrised: two distinct edges of the abstract
graph may have the same image segment. The twin-tree factors
record the special components consisting of two identical branches;
other double-edge configurations occur as well.

\subsubsection{Vertex types}
\label{subsub:vertex-types}

In \cite[Lemma~3.9 and Figure~3]{JMPR25}, the local
appearance of the image $f(\Gamma)$ at a double point is
classified into ten cases~(A)--(J), schematically adapted in
Figure~\ref{fig:vertex-types}.

\begin{figure}[!htbp]
\centering
\resizebox{0.88\textwidth}{!}{\input{figures/combtypes_fatpoint_2}}
\caption{Local building blocks at a double point,
schematically adapted from \cite[Figure~3]{JMPR25}.  The large
filled dot marks the double point; a double line represents a
\emph{double edge} (two edges of~$\Gamma$ of the same primitive
direction mapped onto the same image).}
\label{fig:vertex-types}
\end{figure}
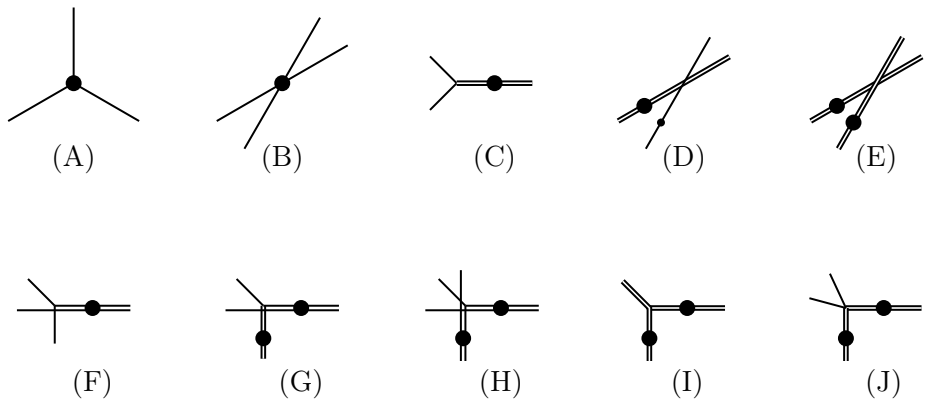

For a merged point not contained in a twin tree, the
corresponding factor in~\eqref{eq:mult} depends on the
vertex type. We group these points into two classes:

\begin{itemize}[leftmargin=2em,nosep]
\item \emph{Type~A merge} (case~(A)):\; the double point
  lies on a $3$-valent vertex~$v$ whose \emph{adjacent
  elevator edge} has weight~$m_v$; the two other edges have
  weight~$1$ in vertically stretched position.  The local
  factor is
  $\gamma(m_v,d_j)\cdot m_v^{\Atrop}$,
  with $\gamma$ as in~\eqref{eq:gamma}.
\item \emph{Type~R merge}:\; the double point lies outside
  the type-A vertices and outside any twin tree (the index
  set~$R$ of~\cite[Definition~4.2]{JMPR25}).  The local factor
  is $\beta_j=\ang{2}+\ang{2d_j}$.
\end{itemize}

\subsubsection{Twin trees}
\label{subsub:twin-trees}

A \emph{double edge} in the image~$f(\Gamma)$ means an image
edge whose preimage contains two edges of the abstract graph
$\Gamma$ mapping to it with the same primitive direction.
Their weights need not be equal; case~(J) of
\cite[Lemma~3.9]{JMPR25} explicitly allows different weights.
A \emph{twin tree}~$\cT$ of $(\Gamma,f)$ is a
connected component of such double edges whose two copies
in~$\Gamma$ are identical parametrisations of the same image,
attached to the rest of the curve at a unique vertex of
type~(C) or a type-(C)-part of~(G)
(\cite[Def.~3.14]{JMPR25}).  Three typical examples from
\cite[Figure~5]{JMPR25} are schematically adapted in
Figure~\ref{fig:twin-trees} below.

\begin{figure}[!htbp]
\centering
\resizebox{0.88\textwidth}{!}{\input{figures/twin_trees}}
\caption{Examples of twin trees with their multiplicities,
schematically adapted from \cite[Figure~5]{JMPR25}.}
\label{fig:twin-trees}
\end{figure}
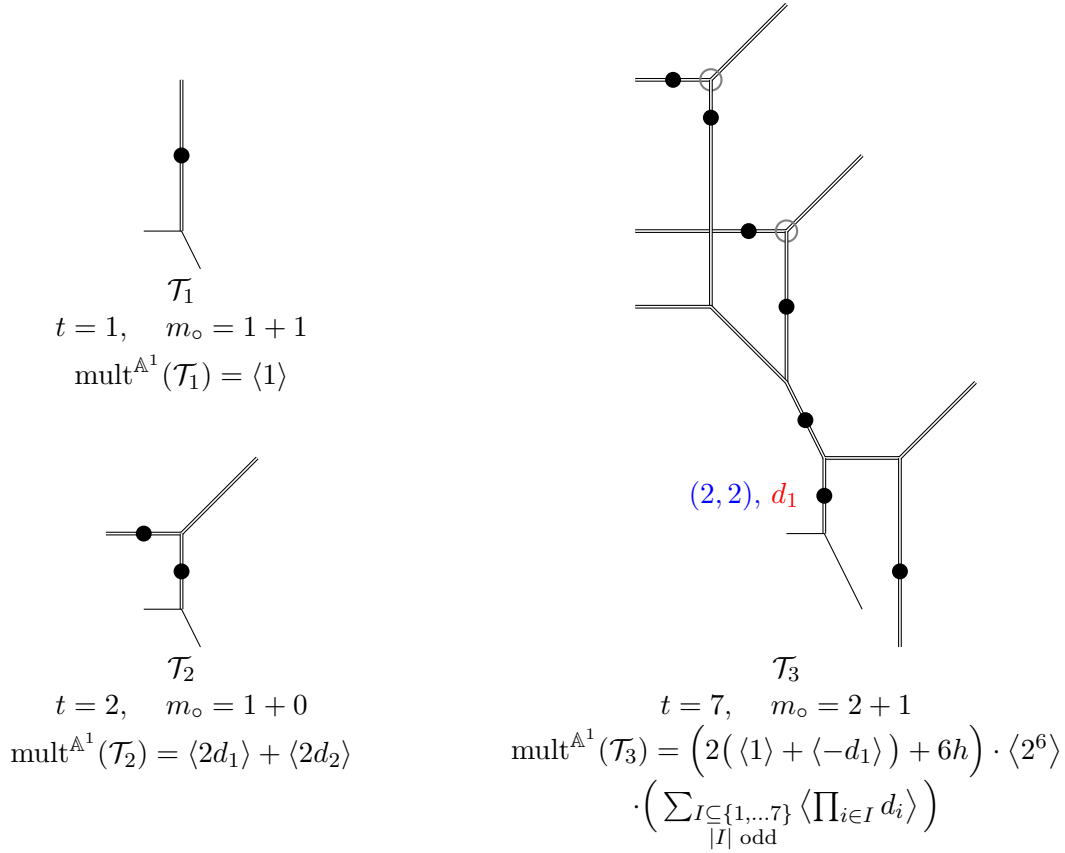

For a twin tree~$\cT$, let $m_\mathrm{root}$ denote the weight
of \emph{each} of the two root edges, not their sum in the image.
Write $J(\cT)\subseteq\{1,\dots,s\}$ for the indices of its
double points and $t=|J(\cT)|$. In the next two formulas only,
relabel these points and their parameters locally as
$q_1,\dots,q_t$ and $d_1,\dots,d_t$; they need not be the first
$t$ global parameters. Put $m_\circ=m_\mathrm{root}+U$, where
$U$ is the number of unbounded twin elevators, counting each
pair once. The twin-tree
multiplicity (\cite[Definition~4.4 and Definition~10.11]{JMPR25}) is
\begin{equation}\label{eq:twin-mult}
  \mult^{\Atrop}(\cT)
  =\prod_{(e,e')\in E_{\mathrm{tw}}(\cT)}
   \mult^{\Atrop}(e,e')\cdot
   \ang{2^{t-1}}\cdot
   \sum_{\substack{I\subseteq\{1,\dots,t\}\\
                   |I|\equiv m_\circ\!\!\pmod 2}}
   \Bigl\langle\prod_{i\in I}d_i\Bigr\rangle,
\end{equation}
where $E_{\mathrm{tw}}(\cT)$ is the set of twin elevators of
\cite[Definition~4.4]{JMPR25}, each pair counted once, including
the root pair. If a marked point subdivides an elevator, its
two pieces are not counted as two separate elevators.
For each such
$(e,e')$ of weight~$m$ carrying a double-point condition with
quadratic parameter~$d$, the twin-edge multiplicity is
\begin{equation}\label{eq:twin-edge-mult}
  \mult^{\Atrop}(e,e')=
  \begin{cases}
    \ang{1}+\dfrac{m^2-1}{2}\bigl(\ang{1}+\ang{-d}\bigr)
    +\dfrac{m^4-m^2}{2}h & m\text{ odd},\\[8pt]
    \dfrac{m^2}{2}\bigl(\ang{1}+\ang{-d}\bigr)
    +\dfrac{m^4-m^2}{2}h & m\text{ even}.
  \end{cases}
\end{equation}

\subsubsection{The full multiplicity formula}
\label{subsub:full-multiplicity}
\label{def:vertex-mult}

The globally grouped multiplicity of a floor diagram~$\cD$ is
\begin{equation}\label{eq:mult}
  \mult^{\Atrop}(\cD)
  =\prod_i\mult^{\Atrop}(\cT_i)\cdot
   \prod_{v\in V_{(A)}}G_A(m_v,d_v)\cdot
   \prod_{j\in R}\beta_j\cdot
   \prod_{e\in E_{\mathrm{el}}^{\neg A,\neg T}}
      \bigl(m_e^{\Atrop}\bigr)^2.
\end{equation}
The products range over twin trees, type-A vertices, type-R points,
and bounded ordinary elevators that are neither incident to a type-A
vertex nor contained in a twin tree. The elevator factors are indexed
on the abstract parametrised curve: distinct non-twin elevators
with overlapping images remain distinct. Each elevator is counted
once, with its two endpoint contributions grouped. Here
$G_A(m_v,d_v)=\gamma(m_v,d_v)m_v^{\Atrop}$ and
$\beta_j=\ang{2}+\ang{2d_j}$. Formula~\eqref{eq:mult} is the
elevator-grouped formula of \cite[Remark~4.10]{JMPR25}, transported
to floor diagrams by \cite[Theorem~10.13]{JMPR25}. Equivalently,
start with \cite[Definition~10.11]{JMPR25} and group the two
factors belonging to each remaining elevator. The square is
attached to this whole elevator, not separately to each edge of
the bipartite floor diagram. At a four-valent non-twin vertex,
the ungrouped factor is $m_v^{\Atrop}(m_v')^{\Atrop}$; its two
factors are assigned to their respective elevators before this
global regrouping. No individual vertex factor is squared.

\subsubsection{Example}
\label{subsub:mini-example}

Using the labels of \cite[Figure~6]{JMPR25}, for the quartic
curve on the left
(five double and one simple point), with $d_1,d_2$ on a
twin tree whose root pair has weights $(1,1)$, and $d_4$ a type-R
merge, the product of their local factors is
\[
  \mult^{\Atrop}(\cT)\cdot\beta_4
  =\bigl(\ang{2d_1}+\ang{2d_2}\bigr)\cdot
   \bigl(\ang{2}+\ang{2d_4}\bigr).
\]
Here the first factor is obtained from~\eqref{eq:twin-mult}
with $t=2$, $m_\mathrm{root}=1$, $U=0$ and $m_\circ=1$;
only the $|I|=1$ terms occur. Analogously, the middle
diagram has the grouped type-A factor
$G_A(3,d_2)=\gamma(3,d_2)3^{\Atrop}
=\ang{1}+\ang{2}+\ang{-2d_2}+3h$.
The vertex factor alone is
$\gamma(3,d_2)=\ang{3}+\ang{6}+\ang{-6d_2}$.
These expressions follow directly
from~\eqref{eq:gamma}, \eqref{eq:twin-mult}, and
\eqref{eq:mult}.

\begin{definition}[Floor-diagram count]\label{def:count}
Fix $r,s\geq0$ with $r+2s=n(\Delta)$ and an ordered tuple
$(d_1,\dots,d_s)$. The \emph{merge positions} record which
of the $r+2s$ abstract positions are paired into
$s$~double points. The $j$-th pair from bottom to top carries
$d_j$; unit shifts preserve this order of labels.
For a merge configuration~$\pi$, let
$\cS_\Delta(r,s;\pi)$ be the corresponding finite set of floor
diagrams compatible with~$\pi$.  The \emph{floor-diagram count
with merge configuration~$\pi$} is
\[
  \Nfloor_{\Delta,\pi}(r,(d_1,\dots,d_s))
  =\sum_{\cD\in\cS_\Delta(r,s;\pi)}
    \mult^{\Atrop}(\cD)\in\GW(k).
\]
When the merge configuration is fixed or irrelevant, we suppress
the~$\pi$ from the subscript.  For two choices~$\pi,\pi'$ of merge
position we write
\[
  \Delta N=\Nfloor_{\Delta,\pi}
  -\Nfloor_{\Delta,\pi'}.
\]
\end{definition}

\begin{lemma}[Merge-position graph and dissolution]
\label{lem:merge-graph}
Encode a merge configuration by a subset
$P=\{p_1<\cdots<p_s\}\subset\{1,\dots,r+2s-1\}$ with
$p_{i+1}\geq p_i+2$, where $p_i$ records the merged adjacent
pair $(p_i,p_i+1)$.  The graph whose edges are unit shifts
$p_i\mapsto p_i\pm1$ preserving this condition is connected.
Moreover, if two configurations differ by a unit shift of a
labelled merged pair and a different labelled pair is
dissolved, the two resulting configurations again differ by
the same unit shift in the problem with $s-1$ double points
and $r+2$ simple points.  If the shifted pair itself is
dissolved, the two resulting configurations are identical.
\end{lemma}

\begin{proof}
For connectivity, slide the leftmost pair until it starts at
position~$1$, then slide the second pair until it starts at
position~$3$, and continue. At step~$i$, the earlier pairs
already start at $1,3,\dots,2i-3$. While $p_i>2i-1$, the start
can slide from $p_i$ to $p_i-1$: it remains at least~$2$
away from the previous start and moves farther from the next.
Repeating these moves brings any configuration to
the canonical one, $\{1,3,\dots,2s-1\}$.

To dissolve a labelled pair, remove it from the matching and
regard its two positions as simple. The ordered set still has
the same $r+2s=(r+2)+2(s-1)$ positions.
A unit shift of another labelled pair is therefore unchanged
by dissolution, whereas removing the shifted pair itself gives
identical remaining matchings on the two sides.
\end{proof}

\section{Algebraic properties of the type-A product}
\label{sec:typeA}

This section computes $\gamma(m,d)\cdot m^{\Atrop}$ and
extracts its $\ang{d}$-dependence.

\begin{proposition}[Type-A product]\label{prop:typeA}
Let $m\geq1$ with $m\cdot1_k\neq0$ and let $d\in k^\times$.
\begin{enumerate}[label=\textup{(\alph*)}]
\item\label{it:even}
  For $m$ even,
  $\gamma(m,d)\cdot m^{\Atrop}=\dfrac{m^2}{2}h$, independent
  of~$d$.
\item\label{it:odd}
  For $m$ odd,
  \begin{equation}\label{eq:odd-product}
    \gamma(m,d)\cdot m^{\Atrop}
    =\ang{1}+\frac{m-1}{2}\bigl(\ang{2}+\ang{-2}\ang{d}\bigr)
     +\frac{m(m-1)}{2}h.
  \end{equation}
  In particular $\gamma(1,d)\cdot 1^{\Atrop}=\ang{1}$,
  and for $m\geq3$ the parameter~$d$ enters only through
  $\ang{d}$, with coefficient $\frac{m-1}{2}\ang{-2}$.
\end{enumerate}
\end{proposition}

\begin{proof}
For~\ref{it:even}, Definition~\ref{def:gamma} gives
$\gamma(m,d)=\tfrac{m}{2}h=m^{\Atrop}$. Hence
\[
  \gamma(m,d)\cdot m^{\Atrop}
  =\frac{m^2}{4}h^2=\frac{m^2}{2}h
\]
by Lemma~\ref{lem:gw-arith}\ref{it:hsq}.

For~\ref{it:odd}, set $c=\frac{m-1}{2}$, so that
\(\gamma(m,d)=\ang{m}+c\ang{2m}+c\ang{-2dm}\) and
$m^{\Atrop}=\ang{m}+ch$. Expanding gives
\begin{align*}
  \gamma(m,d)\cdot m^{\Atrop}
  &=\ang{m}^2+c\ang{2m}\ang{m}+c\ang{-2dm}\ang{m}\\
  &\quad+c\ang{m}h+c^2\ang{2m}h+c^2\ang{-2dm}h.
\end{align*}
Using $\ang{m^2}=\ang{1}$, $\ang{2m^2}=\ang{2}$,
$\ang{-2dm^2}=\ang{-2d}=\ang{-2}\ang{d}$ and $\ang{a}h=h$, we obtain
\begin{align*}
  \gamma(m,d)\cdot m^{\Atrop}
  =\ang{1}+c(\ang{2}+\ang{-2}\ang{d})+(c+2c^2)h.
\end{align*}
Now $c+2c^2=c(1+2c)=\frac{m-1}{2}\cdot m=\frac{m(m-1)}{2}$,
which gives~\eqref{eq:odd-product}.  For $m=1$, $c=0$, so
$\gamma(1,d)\cdot 1^{\Atrop}=\ang{1}$.
\end{proof}

\begin{corollary}[$\ang{d}$-coefficient]\label{cor:d-coeff}
The $\ang{d}$-coefficient of
$\gamma(m,d)\cdot m^{\Atrop}$ is
\[
  \begin{cases}
    \dfrac{m-1}{2}\ang{-2} & m\geq3\text{ odd},\\[6pt]
    0                       & \text{otherwise.}
  \end{cases}
\]
Since $\ang{-2}=h-\ang{2}$, it lies in $\bb{Z}\ang{2}+\bb{Z}h$.
\end{corollary}

\begin{proof}
Immediate from Proposition~\ref{prop:typeA}: the even case
has no $\ang{d}$-term; in~\eqref{eq:odd-product} the only
$d$-dependent term is $\frac{m-1}{2}\ang{-2}\ang{d}$, zero
when $m=1$.
\end{proof}

\begin{lemma}[Elevator squares]\label{lem:elevator-sq}
For every $m\geq1$ with $m\cdot1_k\neq0$,
\begin{equation}\label{eq:sq}
  (m^{\Atrop})^2
  =\begin{cases}
    \ang{1}+\dfrac{m^2-1}{2}h & m\text{ odd},\\[6pt]
    \dfrac{m^2}{2}h           & m\text{ even}.
  \end{cases}
\end{equation}
In particular $(m^{\Atrop})^2\in\bb{Z}\ang{1}+\bb{Z}h$.
\end{lemma}

\begin{proof}
For $m$ even, $(\tfrac{m}{2}h)^2=\tfrac{m^2}{4}h^2
=\tfrac{m^2}{2}h$.
For $m$ odd with $c=\tfrac{m-1}{2}$,
\[
  (\ang{m}+ch)^2=\ang{m^2}+2c\ang{m}h+c^2h^2
  =\ang{1}+(2c+2c^2)h=\ang{1}+\tfrac{m^2-1}{2}h.
  \qedhere
\]
\end{proof}

\begin{corollary}[Universal square-field product]
\label{cor:univ-square-product}
After expanding the two sides by
Proposition~\ref{prop:typeA} and Lemma~\ref{lem:elevator-sq},
and then evaluating the formal variable $\ang{d}$ at
$\ang{1}$ in the fixed lifted normal form,
the identity
\[
  \gamma(m,1)\cdot m^{\Atrop}=(m^{\Atrop})^2
\]
holds in~$\GWuniv$ for every $m\geq1$.
\end{corollary}

\begin{proof}
For $m$ even, both sides are $\frac{m^2}{2}h$.  For $m$ odd,
Proposition~\ref{prop:typeA}\ref{it:odd}, after this formal
evaluation $\ang{d}=\ang{1}$,
gives
\[
  \gamma(m,1)\cdot m^{\Atrop}
  =\ang{1}+\tfrac{m-1}{2}\bigl(\ang{2}+\ang{-2}\bigr)
  +\tfrac{m(m-1)}{2}h.
\]
Using the defining relation $\ang{2}+\ang{-2}=h$ in
$\GWuniv$, this becomes
$\ang{1}+\frac{m^2-1}{2}h$, which is
$(m^{\Atrop})^2$ by Lemma~\ref{lem:elevator-sq}.
\end{proof}

\section{Dissolution and real specialisation}\label{sec:identities}

A square parameter dissolves a double point into two simple
points by~\cite[Proposition~5.4]{JMPR25}. We first record the
local type-A identity, then formulate dissolution with the
labels and order needed in the proof.

\begin{proposition}[Square-field identity]
\label{prop:square-field}
For the vertex factor of Definition~\ref{def:gamma},
$\gamma(m,1)=m^{\Atrop}$ for every $m\geq1$ with
$m\cdot1_k\neq0$.
\end{proposition}

\begin{proof}
At $d=1$, $\ang{-2dm}=\ang{-2m}$.

For odd~$m$,
$\gamma(m,1)=\ang{m}+\tfrac{m-1}{2}(\ang{2m}+\ang{-2m})
=\ang{m}+\tfrac{m-1}{2}h=m^{\Atrop}$,
using Lemma~\ref{lem:gw-arith}~\ref{it:hyp}.

For even~$m$, both $\gamma(m,1)$ and $m^{\Atrop}$ equal $\tfrac{m}{2}h$
by definition.
\end{proof}

\begin{proposition}[Labelled dissolution]
\label{prop:labeled-dissolution}
Fix a merge configuration~$\pi$ with labelled double points
$1,\dots,s$.  Let~$\pi^{\mathrm{diss}}_j$ be the configuration
obtained by replacing the labelled $j$-th merged adjacent pair
by the two underlying adjacent simple positions, keeping the
labels and relative order of all other merged pairs.  Then
\[
  \Nfloor_{\Delta,\pi}(r,(d_1,\dots,d_s))\big|_{d_j=1}
  =
  \Nfloor_{\Delta,\pi^{\mathrm{diss}}_j}
    (r+2,(d_1,\dots,\widehat d_j,\dots,d_s))
  \quad\text{in }\GW(k).
\]
\end{proposition}

\begin{proof}
The assertion is the position-labelled form of
\cite[Proposition~5.4]{JMPR25}. Choose generic vertically stretched
representatives of the ordered conditions. In the close-by
deformation used there, choose the horizontal spread~$\eta$, the
new vertical separation~$\delta$, and the smallest pre-existing
vertical gap~$L$ with
\[
  \eta\ll\delta\ll L.
\]
These two separations can be imposed when choosing representatives
of the floor-diagram order: the first ensures that the deformed
conditions are still vertically stretched, and the second ensures
that no other marked level is crossed. The new simple points thus
take the two adjacent positions previously occupied by the pair.
The other double points keep their labels and their relative order.

We now compare the multiplicities, keeping track also of factors
that change elsewhere in the component. For a type-A merge,
Proposition~\ref{prop:square-field} and
Lemma~\ref{lem:elevator-sq} give
\[
  \gamma(m,1)m^{\Atrop}=(m^{\Atrop})^2,
\]
which is the grouped elevator contribution after the type-A
vertex becomes an ordinary trivalent vertex. For a type-R merge,
the specialisation gives
$\beta_j|_{d_j=1}=2\ang{2}=2\ang{1}$ in~$\GW(k)$ by
Lemma~\ref{lem:two-torsion}; these are the two
local deformations appearing in the proof of
\cite[Proposition~5.4, Case~3]{JMPR25}.

If the point lies on a twin tree~$\cT$, dissolving it destroys
that twin tree. As shown in Case~1 of the same proof, the remaining
points on~$\cT$ now carry type-R factors~$\beta_i$, and its twin
elevators become ordinary elevators. The identity compares the
\emph{whole} former tree contribution with that of its unique
deformation; it does not preserve the other local factors
individually. In particular, setting $d_j=1$ changes the parity
sum in~\eqref{eq:twin-mult} into
\[
  \sum_{I\subseteq J(\cT)\setminus\{j\}}
       \ang{\prod_{i\in I}d_i}
  =\prod_{i\in J(\cT)\setminus\{j\}}(\ang{1}+\ang{d_i}).
\]
Together with the twin-elevator identities in that proof, this
gives precisely the new type-R and ordinary-elevator factors.
The possible four-valent root also splits compatibly into two
trivalent vertices, as checked at the end of the cited proof.

Thus, in each case the specialised multiplicity is the sum of
the multiplicities of its deformations, with the required labels
and order. Summing and applying the floor correspondence
\cite[Theorem~10.13]{JMPR25} gives the stated identity for
$\pi^{\mathrm{diss}}_j$.
\end{proof}

\begin{remark}[Dissolution and the coefficient ring]
\label{rem:formal-dissolution}
The type-A identity
$\gamma(m,1)m^{\Atrop}=(m^{\Atrop})^2$ is an identity already
in~$\GWuniv$ by Corollary~\ref{cor:univ-square-product}.
For type-R and twin-tree factors,
Proposition~\ref{prop:labeled-dissolution} is used after
applying the coefficient map to~$\GW(k)$. In type~R, this uses
$\beta_j|_{d_j=1}=2\ang{2}=2\ang{1}$ in~$\GW(k)$
(Lemma~\ref{lem:two-torsion}); the twin-tree case uses the
component calculation in the preceding proof. Thus formal
specialisation $d_j=1$ is followed by the map to
$\GW(k)$ before the inductive hypothesis is invoked.
\end{remark}

\begin{proposition}[Tropical Welschinger invariance]
\label{prop:broccoli}
Over $k=\bb{R}$, with all double-point parameters
$d_j<0$, one has $\sgn(\Delta N)=0$ for every unit shift of
a merge position.
\end{proposition}

\begin{proof}
By \cite[Proposition~5.13]{JMPR25}, the signature of the
quadratically enriched tropical multiplicity agrees with
Shustin's tropical Welschinger multiplicity~\cite{Shu06}
for real point conditions and conjugate pairs. The comparison of
marking and multiplicity conventions in
\cite[Remark~4.19]{GMS13} identifies this with the tropical
Welschinger count of \cite[Corollaries~5.16--5.17]{GMS13}.
Our degree is toric del Pezzo, all prescribed ends are primitive,
and no ends are fixed, so the hypotheses of these corollaries
apply. The
resulting Welschinger count is independent of the chosen
conditions by \cite[Corollary~5.17]{GMS13}, whose proof is
purely tropical via broccoli curves and bridges.  Therefore
the signatures of the two counts across a unit merge shift
are equal.
\end{proof}

\section{Localisation and \texorpdfstring{$d$}{d}-linearity}
\label{sec:linear}

\begin{proposition}[Localisation]\label{prop:local}
In~\eqref{eq:mult}, the extension~$d_j$ appears only in the
unique factor carrying the $j$-th merged point: either a
type-A factor, a type-R factor, or the twin-tree factor of
the unique twin tree containing that point.  All other
factors are independent of~$d_j$.
\end{proposition}

\begin{proof}
Each twin-tree multiplicity $\mult^{\Atrop}(\cT_i)$ depends
only on extensions~$d_l$ of double points lying on~$\cT_i$,
and twin trees are disjoint.
If the $j$-th merged point is not on a twin tree, then it is
either type~A or type~R and its extension occurs in the
corresponding factor of~\eqref{eq:mult}.  If it is on a twin
tree, the extension occurs in that twin-tree factor and in
no other one.
The remaining bounded-edge weights are combinatorial.
\end{proof}

\begin{corollary}[$d_j$-linearity and coefficient stability]
\label{cor:dj-linear}
The lifted multiplicity expression is multi-affine in the
variables $\ang{d_1},\dots,\ang{d_s}$, with the relations
$\ang{d_j}^2=\ang{1}$.  In particular, for each $j$ there
exist $P_j,R_j\in\GW(k)$ independent of $d_j$ such that
\[
  \Delta N=P_j+\ang{d_j}R_j.
\]
Moreover, after extracting coefficients in any subset of the
variables, the resulting coefficient is still affine-linear
in each remaining variable.
\end{corollary}

\begin{proof}
By Proposition~\ref{prop:local}, $d_j$ enters through a
single carrying factor.  Type-A factors are affine-linear by
Corollary~\ref{cor:d-coeff}, and type-R factors by
$\beta_j=\ang{2}+\ang{2}\ang{d_j}$.  In a twin tree, the
subset sum in~\cite[Definition~4.4]{JMPR25} and the
twin-edge factors~\eqref{eq:twin-edge-mult} are products of
terms affine in the relevant $\ang{d_l}$; using
$\ang{d_l}^2=\ang{1}$, any product is again affine in each
individual variable. The contribution of each diagram is
multi-affine, as is the sum of these contributions.
Extracting a coefficient in one variable selects its affine
part and leaves a multi-affine expression in the other
variables.
\end{proof}

\section{Residual top coefficients}\label{sec:residual}

We compute the coefficient of the monomial containing every
double-point parameter, working modulo~$2$ and the hyperbolic form.
We do this separately for each diagram, then use the result for
the wall-crossing coefficient in Section~\ref{sec:main}.

\begin{definition}[Residual coefficient ring]\label{def:residual-ring}
Set
\[
  \overline{\GW}^{\mathrm{univ}}:=\GWuniv/(2,h),
\]
and let $\varepsilon$ be the image of~$\ang{2}$. Then
\[
  \overline{\GW}^{\mathrm{univ}}
  \simeq \bb{F}_2[\varepsilon]/(\varepsilon^2-1),
  \qquad \{1,\varepsilon\}\text{ is an }\bb{F}_2\text{-basis}.
\]
For $s$ double-point parameters, define
\[
  R_s:=\overline{\GW}^{\mathrm{univ}}[x_1,\ldots,x_s]
  /(x_1^2-1,\ldots,x_s^2-1).
\]
Here $x_j$ is the formal variable representing $\ang{d_j}$ in
the chosen lifted multiplicity, not the image of a field-valued
class under a map from~$\GW(k)$.
If $F\in R_s$, its \emph{top coefficient} is the coefficient of
$x_1\cdots x_s$ in its unique multi-affine expansion.
\end{definition}

\begin{lemma}[Residual local factors]\label{lem:residual-local}
The local factors have the following images in
$\overline{\GW}^{\mathrm{univ}}$.
\begin{enumerate}[label=\textup{(\roman*)},nosep]
\item $(m^{\Atrop})^2=1$ for odd~$m$, and
  $(m^{\Atrop})^2=0$ for even~$m$.
\item For an odd-weight type-A merge,
  \[
    \gamma(m,d)m^{\Atrop}
      =1+\frac{m-1}{2}\varepsilon(1+x)
      \quad\text{in }\overline{\GW}^{\mathrm{univ}}[x]/(x^2-1).
  \]
  For even~$m$ this factor is zero. Thus the coefficient of~$x$
  is~$\varepsilon$ if $m\equiv3\pmod4$, and is zero otherwise.
\item A type-R factor is $\beta=\varepsilon(1+x)$.
\item A twin-edge factor is~$1$ for odd weight and~$0$ for even
  weight. If all twin-edge weights are odd, a twin tree carrying
  $t$ parameters has residual multiplicity
  \[
    \varepsilon^{t-1}
    \sum_{\substack{I\subseteq\{1,\ldots,t\}\\
                    |I|\equiv m_\circ\; (\mathrm{mod}\,2)}}
       \prod_{i\in I}x_i.
  \]
\end{enumerate}
\end{lemma}

\begin{proof}
The first assertion follows from Lemma~\ref{lem:elevator-sq}.
For the second, reduce Proposition~\ref{prop:typeA}: the $h$-terms
vanish, $\ang{-2}=\varepsilon$, and the even case is a multiple
of~$h$. The third assertion follows from
$\beta=\ang{2}+\ang{2}\ang{d}$. For the fourth, reduce the
twin-edge formula~\eqref{eq:twin-edge-mult}; the displayed expression
is then the reduction of~\eqref{eq:twin-mult}.
\end{proof}

\begin{lemma}[Twin-tree top coefficient]\label{lem:twin-top}
The residual multiplicity of a twin tree has top coefficient zero.
\end{lemma}

\begin{proof}
The result is immediate if one twin edge has even weight, so assume
that all twin-edge weights are odd. In particular the root weight
$m_{\mathrm{root}}$ is odd. Let $t$ be the number of double points
on the twin tree and let $U$ be the number of its unbounded twin
elevator pairs. Count ends in one of the two identical branches,
or equivalently in the double image, not in both copies. Extend
the root to one artificial end, as in \cite[\S3.3]{JMPR25}.
If $e$ is the resulting number of ends and $f$ the number of
floors in this branch, \cite[Lemmas~3.12 and~3.17]{JMPR25} give
$t=e-1$ and $e=2f+U+1$, respectively. Consequently
\[
  t=2f+U,\qquad t\equiv U\pmod2.
\]
This is also the parity count used in the proof of
\cite[Proposition~5.13]{JMPR25}. Since
$m_\circ=m_{\mathrm{root}}+U$, it follows that
$m_\circ\equiv1+U\equiv1+t\pmod2$. The full monomial
$x_1\cdots x_t$ could occur in Lemma~\ref{lem:residual-local} only
for $I=\{1,\ldots,t\}$; its parity is opposite to $m_\circ$, so
that summand is absent.
\end{proof}

\begin{proposition}[Residual top-coefficient rigidity]
\label{prop:residual-rigidity}
For every floor diagram with $s$ double-point parameters, the top
coefficient of its residual multiplicity belongs to
\[
  \{0,\varepsilon^s\}\subset\overline{\GW}^{\mathrm{univ}}.
\]
Consequently, the top coefficient of the difference between the
residual expressions of any two floor counts lies in
$\bb{F}_2\varepsilon^s$.
\end{proposition}

\begin{proof}
By Proposition~\ref{prop:local}, each parameter occurs in one
carrying factor. If the diagram contains a twin tree, the variables
carried by that tree occur nowhere else, and Lemma~\ref{lem:twin-top}
makes the full top coefficient zero. Otherwise every parameter is
carried by a type-A or type-R factor. By
Lemma~\ref{lem:residual-local}, every non-zero parameter coefficient
is~$\varepsilon$, while every remaining non-zero elevator square
is~$1$. Hence a non-zero top coefficient is the product of exactly
$s$ copies of~$\varepsilon$. Taking the difference of two diagram
sums gives the final assertion.
\end{proof}

\section{Proof of the main theorem}\label{sec:main}

\begin{theorem}[=\,Theorem~\ref{thm:main-intro}]
\label{thm:main}
Under Setting~\ref{set:JMPR}, the count
$\Nfloor_\Delta(r,(d_1,\dots,d_s))$ is independent of the
merge positions, with the ordered tuple of double-point labels
fixed.
\end{theorem}

\begin{proof}[Proof of Theorem~\ref{thm:main}]
Merge configurations are matchings of size~$s$ in the
ordered set of $r+2s$ point positions.  The graph of such
matchings is connected under unit slides of one matched pair
to a neighbouring free pair by Lemma~\ref{lem:merge-graph}.
It therefore suffices to show
$\Delta N=0$ for each unit shift $p\to p+1$ of a single
merge position.  We proceed by induction on~$s$
simultaneously over all~$\Delta$ and~$k$.

\smallskip\noindent\emph{Base case $s=0$.}
With no double points there is no merge position to shift:
the count is trivially position-independent.

\smallskip\noindent\emph{Inductive step.}
Assume the theorem holds for $\leq s-1$ double points.
Fix~$s$ double points and a unit shift of one merge. For this
proof only, rename the parameters so that the shifted pair
carries~$d_s$; this is a change of notation, not a permutation
of the marked conditions or their geometric order.

\medskip\noindent\textsc{Step~1. First factor.}
By Corollary~\ref{cor:dj-linear} with $j=s$,
\begin{equation}\label{eq:lin-s}
  \Delta N=P_s+\ang{d_s}R_s.
\end{equation}
Set $d_s=1$.  By
Proposition~\ref{prop:labeled-dissolution}, the count
becomes that of a problem with $s-1$ double points and
$r+2$ simple points.  Since the shifted pair is precisely
the one being dissolved, Lemma~\ref{lem:merge-graph} says
that the two dissolved merge configurations are identical;
in particular $\Delta N|_{d_s=1}=0$.
From~\eqref{eq:lin-s}, $0=P_s+R_s$, whence $P_s=-R_s$ and
\begin{equation}\label{eq:one-factor}
  \Delta N=R_s(\ang{d_s}-\ang{1}).
\end{equation}

\medskip\noindent\textsc{Step~2. Dissolution cascade.}
If $s=1$, Step~1 is already the full factorisation, with
$C=R_s$.  Assume from now on that $s\ge2$.
We now index the cascade by the number~$q$ of factors already
extracted.  Suppose at stage~$q$, with $1\le q<s$, that
\begin{equation}\label{eq:partial}
  \Delta N=R^{(q)}\cdot
    \prod_{\ell\in S_q}(\ang{d_\ell}-\ang{1}),
\end{equation}
with $S_q\subset\{1,\dots,s\}$, $|S_q|=q$, $s\in S_q$, and
$R^{(q)}$ depending only on the parameters
$\{d_\ell:\ell\notin S_q\}$.
The base case is~\eqref{eq:one-factor}, namely
$q=1$, $S_1=\{s\}$, and $R^{(1)}=R_s$.

Choose a label $a_q\notin S_q$; for definiteness take the
smallest such label.  By Corollary~\ref{cor:dj-linear}, write
\begin{equation}\label{eq:Rj-linear}
  R^{(q)}=R_0+\ang{d_{a_q}}R_1,
\end{equation}
with $R_0,R_1$ independent of~$d_{a_q}$ and of
$\{d_\ell:\ell\in S_q\}$.

Set $d_{a_q}=1$.  By Proposition~\ref{prop:labeled-dissolution}
and the compatibility statement in Lemma~\ref{lem:merge-graph}, the
specialised wall-crossing is the wall-crossing for a problem
with $s-1$ double points and $r+2$ simple points.  The
inductive hypothesis gives $\Delta N|_{d_{a_q}=1}=0$, so
\[
  (R_0+R_1)\prod_{\ell\in S_q}(\ang{d_\ell}-\ang{1})=0
  \qquad\text{for all remaining parameters.}
\]
We use the vanishing of this whole product, without cancelling
any factor.
Substituting~\eqref{eq:Rj-linear} into~\eqref{eq:partial} gives
\begin{align*}
  \Delta N
  &=(R_0+\ang{d_{a_q}}R_1)
    \prod_{\ell\in S_q}(\ang{d_\ell}-\ang{1})\\
  &=\underbrace{(R_0+R_1)
    \prod_{\ell\in S_q}(\ang{d_\ell}-\ang{1})}_{=0}
   +R_1(\ang{d_{a_q}}-\ang{1})
    \prod_{\ell\in S_q}(\ang{d_\ell}-\ang{1}).
\end{align*}
Setting $S_{q+1}=S_q\cup\{a_q\}$ and $R^{(q+1)}=R_1$
yields~\eqref{eq:partial} with~$q$ replaced by~$q+1$.
After the stage $q=s-1$ we obtain
\begin{equation}\label{eq:full}
  \Delta N=C\cdot\prod_{l=1}^{s}(\ang{d_l}-\ang{1}),
  \qquad C\in\GW(k),
\end{equation}
with $C$ independent of all~$d_l$.

\medskip\noindent\textsc{Step~3. Universal top coefficient.}
We construct a lift of~$C$ to~$\GWuniv$ whose integer
coordinates depend only on the combinatorial data.

\smallskip\noindent\emph{(a) Lifting the local factors.}
Consider the extended coefficient ring
\[
  \widetilde{\GW}
  \coloneqq\GWuniv[\ang{d_1},\dots,\ang{d_s}]\big/
  \bigl(\ang{d_l}^2-1\bigr)_{l=1}^{s},
\]
a free $\bb{Z}$-module of rank $3\cdot2^s$ with basis
$\{g\cdot\!\prod_{l\in I}\ang{d_l}:
  g\in\{\ang{1},\ang{-1},\ang{2}\},\,I\subseteq\{1,\dots,s\}\}$.
The symbols $\ang{d_l}$ in this presentation are formal
variables; they are not yet classes in a field-valued ring.
The canonical ring homomorphism
$\widetilde{\varphi_k}\colon\widetilde{\GW}\to\GW(k)$ sending
$\ang{a}\mapsto\ang{a}_k$ and
$\ang{d_l}\mapsto\ang{d_l}_k$ is well-defined because the
imposed relations $\ang{d_l}^2=1$,
$\ang{2}+\ang{-2}=h$ hold in every~$\GW(k)$.

Each multiplicative building block of $\mult^{\Atrop}(\cD)$ is
lifted by the fixed normal form to~$\widetilde{\GW}$:
\begin{itemize}[nosep]
\item \emph{Elevator squares
  $(m^{\Atrop})^2$}
  (formula~\eqref{eq:mult} and
  Lemma~\ref{lem:elevator-sq}): after regrouping the two
  edge factors of each non-type-A, non-twin elevator, this
  is an integer combination of
  $\ang{1}$ and~$h$, so in $\bb{Z}\ang{1}+\bb{Z}h$.
\item \emph{Type-A merge factors
  $\gamma(m,d_l)\cdot m^{\Atrop}$}
  (Proposition~\ref{prop:typeA}):
  for $m$ even, $\tfrac{m^2}{2}h\in\bb{Z}h$;
  for $m$ odd, $\ang{1}+\tfrac{m-1}{2}\bigl(\ang{2}
  +\ang{-2}\ang{d_l}\bigr)+\tfrac{m(m-1)}{2}h$, i.e.\ in
  $\bb{Z}\ang{1}+\bb{Z}\ang{2}+\bb{Z}\ang{-2}\ang{d_l}+\bb{Z}h$.
  In both cases the symbols $\ang{m}$ for
  $m\not\in\{\pm1,\pm2\}$ have already cancelled via
  $\ang{m}^2=\ang{1}$ inside the multiplication
  (cf.~\eqref{eq:odd-product}).
\item \emph{Type-R merge factors}
  $\beta_j=\ang{2}+\ang{2}\ang{d_j}$: in
  $\bb{Z}\ang{2}+\bb{Z}\ang{2}\ang{d_j}$.
\item \emph{Twin-tree factors}
  (\cite[Def.~4.4]{JMPR25}):
  \[
  \mult^{\Atrop}(\cT)
    =\prod_{(e,e')}\mult^{\Atrop}(e,e')\cdot\ang{2^{t-1}}
     \sum_{|I|\equiv m_\circ\!\!\bmod 2}
       \ang{\prod_{i\in I}d_i}.
  \]
  Here the parameters are locally indexed by the double points
  of~$\cT$, as in Section~\ref{subsub:twin-trees}.
  Each twin-edge factor~\eqref{eq:twin-edge-mult} lies in
  $\bb{Z}\ang{1}+\bb{Z}\ang{-d_i}+\bb{Z}h
  \subset\bb{Z}\ang{1}+\bb{Z}h+\bb{Z}\ang{-1}\ang{d_i}$,
  and the sum runs over products
  $\prod_{i\in I}\ang{d_i}$; the whole is in $\widetilde{\GW}$.
\end{itemize}

Hence $\mult^{\Atrop}(\cD)\in\GW(k)$ is the image
under~$\widetilde{\varphi_k}$ of an element
$\widetilde{\mult}(\cD)\in\widetilde{\GW}$ depending only on
the combinatorial data of~$\cD$.

\smallskip\noindent\emph{(b) Extracting the universal coefficient.}
Let $\widetilde{\Delta N}\in\widetilde{\GW}$ be the
difference of the two sums of lifted diagram multiplicities.
The chosen multi-affine presentation has a well-defined
coefficient of the monomial containing all variables~$\ang{d_l}$.
Denote this \emph{top coefficient} by
\begin{equation}\label{eq:C-lift}
  \widetilde{C}\in\GWuniv.
\end{equation}

The coefficient just defined is the coefficient extracted by
the cascade.  Apply Lemma~\ref{lem:multiaffine-cascade} to
$\widetilde{\Delta N}$, with the same order of variables as in
Steps~1--2.  At each stage the specialisation term~$S_q$ of
that lemma is the expression obtained by setting the relevant
$d_j$ equal to~$1$.  After applying
$\widetilde{\varphi_k}$, Remark~\ref{rem:formal-dissolution}
and Lemma~\ref{lem:merge-graph} identify the corresponding
summand
$S_q\prod_{p<q}(\ang{d_{j_p}}-\ang{1})$
with the wall-crossing for the dissolved problem, after the terms
discarded at earlier stages have vanished. The earlier identities
are valid for every choice of the remaining parameters, including
the new specialisation. By the inductive hypothesis, the displayed
summand is zero in~$\GW(k)$. The lemma identifies the coefficient
multiplying $\prod_l(\ang{d_l}-\ang{1})$ in this field-valued
factorisation as
\[
  C_k=\varphi_k(\widetilde{C}).
\]
Injectivity of $\varphi_k$ is not assumed: the multi-affine
coefficient is taken in the fixed presentation
$\widetilde{\Delta N}$, whereas $\GW(k)$ is used only to
establish the displayed vanishing identities.

\smallskip\noindent\emph{(c) Universal coordinates.}
Lemma~\ref{lem:GWuniv-props}\ref{it:GWuniv-basis} gives the
$\bb{Z}$-basis $\{\ang{1},h,\ang{2}\}$.  Therefore there are
unique integers $n_1,n_2,\mu$ such that
\begin{equation}\label{eq:C-struct}
  \widetilde{C}=n_1\ang{1}+n_2\ang{2}+\mu h\in\GWuniv,
  \qquad n_1,n_2,\mu\in\bb{Z},
\end{equation}
and
$\varphi_k(\widetilde{C})=n_1\ang{1}_k+n_2\ang{2}_k+\mu h_k$.

\medskip\noindent\textsc{Step~4. Real specialisation.}
Apply the signature $\sgn\colon\GW(\bb{R})\to\bb{Z}$ to
$\Delta N$, specialising to $k=\bb{R}$ with every $d_l<0$.
Then $\ang{d_l}_\bb{R}=\ang{-1}_\bb{R}$ and
$\sgn(\ang{d_l}-\ang{1})=-2$.  Over~$\bb{R}$,
$\sgn(\ang{1})=\sgn(\ang{2})=1$ and $\sgn(h)=0$, so
$\sgn(\varphi_\bb{R}(\widetilde{C}))=n_1+n_2$, whence
\[
  \sgn(\Delta N)=(n_1+n_2)\cdot(-2)^s.
\]
By Proposition~\ref{prop:broccoli}, $\sgn(\Delta N)=0$, and
therefore
\begin{equation}\label{eq:C-after-brocc}
  n_2=-n_1,\qquad
  \widetilde{C}=n_1\bigl(\ang{1}-\ang{2}\bigr)+\mu h\in\GWuniv
  \quad(n_1,\mu\in\bb{Z}).
\end{equation}

Applying $\varphi_k$ and using
$h(\ang{d_l}-\ang{1})=0$
(Lemma~\ref{lem:gw-arith}\ref{it:annihilate}) to kill the
$\mu h$ term, we obtain
\begin{equation}\label{eq:reduced-difference}
  \Delta N
  =n_1\bigl(\ang{1}-\ang{2}\bigr)\prod_{l=1}^{s}
    \bigl(\ang{d_l}-\ang{1}\bigr)
  \quad\text{in }\GW(k).
\end{equation}

\medskip\noindent\textsc{Step~5. Residual parity and conclusion.}
Let $\overline C$ denote the image of the universal top coefficient
$\widetilde C$ in $\overline{\GW}^{\mathrm{univ}}$.
Proposition~\ref{prop:residual-rigidity} gives
\[
  \overline C\in\bb{F}_2\varepsilon^s.
\]
On the other hand, Step~4 gives $n_2=-n_1$, so
\[
  \overline C=n_1\ang{1}+n_2\varepsilon
             =n_1(1+\varepsilon)\in\bb{F}_2(1+\varepsilon).
\]
Since $\{1,\varepsilon\}$ is an $\bb{F}_2$-basis and
$\varepsilon^s\in\{1,\varepsilon\}$, the lines
$\bb{F}_2\varepsilon^s$ and $\bb{F}_2(1+\varepsilon)$ intersect
only at zero. Hence $\overline C=0$ and $n_1\equiv0\pmod2$.
Write $n_1=2q$. Substituting in~\eqref{eq:reduced-difference} and using
$2(\ang{1}-\ang{2})=0$ from Lemma~\ref{lem:two-torsion} gives
$\Delta N=0$ in~$\GW(k)$ for every admissible field~$k$, closing
the induction.

\end{proof}

\section*{Acknowledgements}
This paper grew out of a problem suggested by Hannah Markwig during my
M1 internship at the École normale supérieure de Paris. I am grateful
to her for proposing the question and for her guidance. I thank Ilia
Itenberg for his advice and for his course on enumerative geometry at
Jussieu, and Emmanuel Giroux for encouraging me to write the paper. I
also thank Erwan Brugallé, Johannes Rau, and Kirsten Wickelgren for
valuable comments and discussions.

\end{document}

%% file: figures/drawing_tools.tex
\def\vertexsize {1.5pt}

\def\doubleedgesep {0.03}

\newcommand{\thinpoint}[1]{\fill (#1) circle [radius = \vertexsize];}
\newcommand{\fatpoint}[1]{\fill (#1) circle [radius = 2*\vertexsize];}

\newcommand{\edgeweight}[1]{{\color{blue}#1}}

\newcommand{\quadextension}[1]{{\color{red}#1}}

%% file: figures/combtypes_fatpoint_2.tex
\tikzset{every picture/.style={line width=0.75pt}}
\setlength\tabcolsep{0.5cm}
\begin{tabular}{c c c c c}
    \begin{tikzpicture}
        \path[draw] (0,0) --+ (90:1);
        \path[draw] (0,0) --+ (210:1);
        \path[draw] (0,0) --+ (-30:1);
        \fatpoint{0, 0}
        \draw (0, -1) node {\crtcrossreflabel{(A)}[fatPointOnVertex]};
    \end{tikzpicture}
    &
    \begin{tikzpicture}
        \path[draw] (30:-1) -- (30:1);
        \path[draw] (60:-1) -- (60:1);
        \fatpoint{0, 0}
        \draw (0, -1) node {\crtcrossreflabel{(B)}[fatPointOnParallelogram]};
    \end{tikzpicture}
    &
    \begin{tikzpicture}
        \path[draw] (0,0) -- (135:0.5);
        \path[draw] (0,0) -- (-135:0.5);
        \path[draw, double] (0,0) -- (1,0);
        \fatpoint{0.5, 0}
        \draw (0.5, -1) node {\crtcrossreflabel{(C)}[fourValentVertex]};
    \end{tikzpicture}
    &
    \begin{tikzpicture}
        \path[draw, double] (30:-1) -- (30:0.7);
        \path[draw] (60:-1) -- (60:0.7);
        \fatpoint{30:-0.6}
        \thinpoint{60:-0.6}
        \draw (0, -1) node {\crtcrossreflabel{(D)}[parallelogramWithOneDoubleEdge]};
    \end{tikzpicture} 
    &
    \begin{tikzpicture}
        \path[draw, double] (30:-1) -- (30:0.7);
        \path[draw, double] (60:-1) -- (60:0.7);
        \fatpoint{30:-0.6}
        \fatpoint{60:-0.6}
        \draw (0, -1) node {\crtcrossreflabel{(E)}[parallelogramWithTwoDoubleEdges]};
    \end{tikzpicture}
    \\[1cm]
    \begin{tikzpicture}
        \path[draw] (0,\doubleedgesep) --+ (135:0.5);
        \path[draw] (0,\doubleedgesep) --+ (-90:0.5);
        \path[draw] (0,\doubleedgesep) -- (1,\doubleedgesep);
        \path[draw] (-0.5, -\doubleedgesep) -- (1, -\doubleedgesep);
        \fatpoint{0.5, 0}
        \draw (0.5, -1) node {\crtcrossreflabel{(F)}[triangleWithMergedEdge]};
    \end{tikzpicture}
    &
    \begin{tikzpicture}
        \path[draw] (0,\doubleedgesep) --+ (135:0.5);
        \path[draw, double] (0,\doubleedgesep) --+ (-90:0.7);
        \path[draw] (0,\doubleedgesep) -- (1,\doubleedgesep);
        \path[draw] (-0.5, -\doubleedgesep) -- (1, -\doubleedgesep);
        \fatpoint{0.5, 0}
        \fatpoint{0, -0.4}
        \draw (0.5, -1) node {\crtcrossreflabel{(G)}[fourValentVertexWithMergedEdge]};
    \end{tikzpicture}
    &
    \begin{tikzpicture}
        \path[draw] (\doubleedgesep,\doubleedgesep) --+ (135:0.5);
        \path[draw] (\doubleedgesep,\doubleedgesep) -- (\doubleedgesep, -0.7);
        \path[draw] (\doubleedgesep,\doubleedgesep) -- (1,\doubleedgesep);
        \path[draw] (-0.5, -\doubleedgesep) -- (1, -\doubleedgesep);
        \path[draw] (-\doubleedgesep, -0.7) -- (-\doubleedgesep, 0.5);
        \fatpoint{0.5, 0}
        \fatpoint{0, -0.4}
        \draw (0.5, -1) node {\crtcrossreflabel{(H)}[triangleWithTwoMergedEdges]};
    \end{tikzpicture}
    &
    \begin{tikzpicture}
        \path[draw, double] (0,0) -- (1, 0);
        \path[draw, double] (0,-0.7) -- (0, 0) --+ (135:0.5);
        \fatpoint{0.5, 0}
        \fatpoint{0, -0.4}
        \draw (0.5, -1) node {\crtcrossreflabel{(I)}[allDoubleVertex]};
    \end{tikzpicture}
    &
    \begin{tikzpicture}
        \path[draw, double] (0,0) -- (1, 0);
        \path[draw, double] (0,-0.7) -- (0, 0);
        \path[draw] (0,0) --+ (115:0.5);
        \path[draw] (0,0) --+ (165:0.5);
        \fatpoint{0.5, 0}
        \fatpoint{0, -0.4}
        \draw (0.5, -1) node {\crtcrossreflabel{(J)}[mergedTriangles]};
    \end{tikzpicture}
\end{tabular}

%% file: figures/twin_trees.tex
\begin{tikzpicture}[line width=0.75pt, double distance=1pt]
			\path[draw, double] (-2, 4) --++ (1, 0) --++ (1, 1);
			\fatpoint{-1.5, 4}
	        \draw[thick, gray] (-1, 4) circle (4pt);
			
			\path[draw, double] (-2, 2) --++ (2, 0) --++ (1, 1);
			\fatpoint{-0.5, 2}
	        \draw[thick, gray] (0, 2) circle (4pt);
			
			\path[draw, double] (-1, 4) --++ (0, -3);
			\fatpoint{-1, 3.5}
			
			\path[draw, double] (0, 2) -- ++ (0, -2);
			\fatpoint{0, 1}
			
			\path[draw, double] (-2, 1) --++ (1, 0) --++ (1, -1) --++ (0.5, -1) --++ (1, 0) -- ++ (1,1);
			\fatpoint{0.25, -0.5}
		
		\path[draw, double] (0.5, -1) --++ (0, -1); 
		\fatpoint{0.5, -1.5}
		\draw (0.3, -1.5) node[anchor = east] {\edgeweight{$(2, 2)$, \quadextension{$d_1$}}};
		
		\path[draw, double] (1.5, -1) --++ (0, -2.5);
		\fatpoint{1.5, -2.5}
		
			\path[draw] (0, -2) --++ (0.5, 0) --++ (0.5, -1);

		\draw (0, -3.5) node[anchor = north] {$\cT_3$};
        \draw (0, -4) node[anchor = north] {$t = 7$, \quad $m_\circ = 2 + 1$};
        \draw (0, -4.4) node[anchor = north, align = center] {$\mult^{\AA^1}(\cT_3) = \Big( 2\big(\gw{1} + \gw{-d_1} \big) + 6\h \Big) \cdot \gw{2^6}$ \\ $\displaystyle\cdot \Big( \sum_{\substack{I \subseteq \{1, \ldots, 7\}  \\ |I| \text{ odd}}} \gw{\prod_{i \in I} d_i}\Big)$};
		
		\begin{scope}[shift={(-8, 3)}]
			\path[draw, double] (0, 1) -- ++ (0, -2);
			\fatpoint{0, -0}
			\path[draw] (-0.5, -1) --++ (0.5,0) --++ (0.25, -0.5);
			
			\draw (0, -1.5) node[anchor = north] {$\cT_1$};
            \draw (0, -2) node[anchor = north] {$t = 1$, \quad $m_\circ = 1 + 1$};
            \draw (0, -2.5) node[anchor = north] {$\mult^{\AA^1}(\cT_1) = \gw{1}$};
		\end{scope}
	
	\begin{scope}[shift={(-8, -2)}]
		\path[draw,double] (-1, 0) --++ (1, 0) --++ (1,1);
		\fatpoint{-0.5, 0}
		\path[draw, double] (0, 0) -- ++ (0, -1);
		\fatpoint{0, -0.5}
		\path[draw] (-0.5, -1) --++ (0.5,0) --++ (0.25, -0.5);
		
		\draw (0, -1.5) node[anchor = north] {$\cT_2$};
        \draw (0, -2) node[anchor = north] {$t = 2$, \quad $m_\circ = 1 + 0$};
        \draw (0, -2.5) node[anchor = north] {$\mult^{\AA^1}(\cT_2) = \gw{2d_1} + \gw{2 d_2}$};
	\end{scope}
		
	\end{tikzpicture}